\documentclass[reqno]{amsart}

\usepackage{amssymb}
\usepackage{enumitem}
\usepackage{hyperref}
\usepackage{xcolor}

\newtheorem{thm}{Theorem}[section]

\newtheorem{cor}[thm]{Corollary}
\newtheorem{lem}[thm]{Lemma}
\newtheorem{prop}[thm]{Proposition}
\theoremstyle{definition}
\newtheorem{defn}[thm]{Definition}

\newtheorem{ex}[thm]{Example}
\newtheorem{rem}[thm]{Remark}

\numberwithin{equation}{section}

\newcommand{\nbb}{\mathbb{N}}
\newcommand{\cbb}{\mathbb{C}}
\newcommand{\rbb}{\mathbb{R}}

\newcommand{\ccal}{\mathcal{C}}

\newcommand{\pcal}{\mathcal{P}}
\newcommand{\USC}{\mathcal{USC}}
\newcommand{\qbrem}{$q$-Bremermann}

\newcommand{\ph}{pluri\-harmonic}

\newcommand{\psh}{pluri\-sub\-harmonic}

\newcommand{\qpsh}{$q$-pluri\-sub\-harmonic}
\newcommand{\qpshy}{$q$-pluri\-sub\-harmonicity}

\newcommand{\sqpsh}{strictly $q$-pluri\-sub\-harmonic}

\newcommand{\PSH}{\mathcal{PSH}}

\newcommand{\spsc}{strictly pseudo\-convex}
\newcommand{\qpsc}{$q$-pseudo\-convex}

\newcommand{\spscy}{strict pseudoconvexity}

\newcommand{\cl}[1]{\overline{#1}}
\newcommand{\nbh}{neighbourhood}

\newcommand{\cont}{continuous}
\newcommand{\usc}{upper semi-continuous}
\newcommand{\lsc}{lower semi-continuous}
\newcommand{\fct}{function}
\newcommand{\fcts}{functions}

\renewcommand{\and}{\quad \mathrm{and} \quad}

\newcommand{\cld}{\overline{D}}

\newcommand{\relc}{\Subset}

\newcommand{\D}{\displaystyle}
\newcommand{\eps}{\varepsilon}
\newcommand{\vphi}{\varphi}

\title[Perron-Bremermann envelope for $q$-plurisubharmonic functions]{The Perron-Bremermann envelope for $q$-plurisubharmonic functions on unbounded domains in $\cbb^n$}

\author{Pawlaschyk, Thomas}



\begin{document}
\begin{abstract}

Let $D$ be an unbounded domain in $\mathbb{C}^n$, and let $f$ be a bounded continuous function prescribed on the boundary of $D$. We show that if $D$ has $r$-peak points on its boundary and is of bounded type, $f$ extends to a maximal bounded continuous function $F$ on $\overline{D}$ that is $q$-plurisubharmonic and $(n-q-1)$-plurisuperharmonic (i.e., $(dd^c F)^n=0$) on $D$, and that coincides with the Perron-Bremermann envelope created with respect to bounded $q$-plurisubharmonic functions on $\overline{D}$.

\end{abstract}

\maketitle

\section*{Acknowledgements}

I developed the main ideas of this paper during my research stay in 2016 at POSTECH in Pohang, South Korea. I would like to thank Kang-Tae Kim for his hospitality during this stay and his friendship developed from many more stays over the years. I would also like to thank my doctoral supervisor and colleague Nikolay Shcherbina for having supported my academical career. He passed away on February 7, 2026, after a long, severe illness.

\section{Introduction}\label{intro}

The Dirichlet-Bremermann problem for \psh\ \fcts\ is of special interest in pluripotential theory and several complex variables. It was solved by Bremermann~\cite{Br} for bounded \spsc\ domains with continuous boundary data using the Perron envelope. Continuity of the solution is due to Walsh~\cite{Wa}. Bedford and Taylor~\cite{BT} showed that it is the unique solution satisfying $(dd^c F)^n=0$. Recent developments on bounded domains can be found in~\cite{Nil} and~\cite{NTT}. For unbounded domains, additional assumptions are required. Simioniuc and Tomassini~\cite{SiTo} presented a constructive method for obtaining a continuous solution on strictly convex and  special strictly pseudoconvex domains. Nguyen and Hung~\cite{NH} proved results on unbounded $B$-regular domains and domains of so-called \textit{bounded type} using Jensen measures.

In this note, we study the Dirichlet-Bremermann problem for \textit{\qpsh}\ \fcts\ in the sense of~\cite{HM} from the viewpoint of the \textit{$q$-Perron-Bremermann envelope} given by
\[
\pcal_{f,q,D,M}(z)=\sup\{\psi(z) : \psi \in \PSH_q(D) \cap{\ccal(\cld)}, \ \psi \leq f \ \text{on}\ \partial D, \ \psi\leq M \ \text{on}\ \cld\},
\]
where $D$ is a possibly unbounded domain in $\cbb^n$, $f$ is a is continuous \fct\ on $\partial D$, and $M$ is a constant with $\sup_{\partial D} f \leq M \leq +\infty$. The aim is to find a maximal \cont\ \qpsh\ extension of $f$ on $\cld$.

Hunt-Murray \cite{HM} solved the issue on bounded strictly \qpsc\ domains, extending \cont\ boundary data to a so-called \textit{$q$-Bremermann} \fct\ on $\cld$, i.e., a \fct\ that is \cont\ on $\cld$, and \qpsh\ and $(n-q-1)$-plurisuperharmonic on $D$. Kalka \cite{Ka} presented a more general version on strictly $r$-pseudoconvex domains. S\l{}odkowski~\cite{Sl84} developed approximation techniques for \qpsh\ \fcts\ and deduced the uniqueness of the solution. Simioniuc and Tomassini~\cite{SiTo} extended their results from \psh\ \fcts\ ($q=0$) to the \qpsh\ setting on unbounded strictly convex domains.

In the first section, we establish basic properties of $q$-Bremermann and $q$-maximal functions. In the second one, we study the $q$-Perron-Bremermann envelope introduced above. In the third and final section, we present classical results on the $q$-Dirichlet-Bremermann problem and our main Theorem~\ref{thm-q-Dirichlet-unbounded}. We prove that $D$ admits a unique maximal continuous extension of $f$ to $\cld$ that is $q$-Bremermann on $D$, provided that $D$ has $r$-peak points ($0 \leq r \leq q \leq n-r-1$) on its boundary and is of bounded type.

\section{Properties of q-Bremermann functions}\label{sec-prop-q-brem}

We begin by listing main properties of \qbrem\ \fcts. Let $A$ be a subset of $\cbb^n$. We write $\USC(A)$ for the set of \usc\ \fcts\ $u:A\to[-\infty,+\infty)$, and $\ccal(A)$ for the set of real-valued \cont\ \fcts\ on $A$. Unless stated otherwise, $D$ denotes a domain in $\cbb^n$, possibly unbounded. $B_r(p)$ is used for the ball centered at $p \in \cbb^n$ with radius $r>0$. By $U \relc A$ we mean that $\cl{U}$ is a compact subset of $A$.

\begin{defn} Let $f:A\to\rbb$ be a real-valued \fct\ on an arbitrary subset $A$ of~$\cbb^n$. By $f^*(z):=\limsup_{\zeta \to z} f(\zeta)$ we mean the \textit{upper semi-continuous regularization} of $f$ on $A$, and by $f_*(z):=\liminf_{\zeta \to z} f(\zeta)$ the \textit{lower semi-continuous regularization} of $f$ on $A$. We denote by $\USC(A)$ the set of all \usc\ \fcts\ on $A$, and by $\ccal(A)$ the set of all \cont\ \fcts\ on $A$.
\end{defn}

\begin{rem} $f^*$ is \usc\ and $f_*$ is \lsc, satisfying $f_* \leq f \leq f^*$ on $A$. Furthermore, $(-f)^*=-f_*$. The \fct\ $f$ is \usc\ on $A$ if and only if $f^*=f$. It is \lsc\ on $A$ if and only if $f_*=f$. Therefore, $f \in \ccal(A)$ if and only if $f_*=f=f^*$. 
\end{rem}

The following notion of \qpshy\ is in the sense of \cite{HM}.

\begin{defn} \label{def-qpsh}\label{defn-sph} Let $q\in\{ 0,\ldots,n-1\}$ and let $\psi:U\to[-\infty,+\infty)$ be an \usc\ \fct\ on an open set $U$ in $\cbb^n$.
\begin{enumerate}



\item The \fct\ $\psi$ is \emph{\qpsh}\ on $U$ if for every complex $(q+1)$-dimensional affine subspace $\Pi$, every ball $B \relc U$ and every \ph\ \fct\ $h$ on defined on an open neighborhood of $\cl{B}$ with $\psi \leq h$ on $\partial B \cap \Pi$ we have that $u \leq h$ on $\cl{B}\cap \Pi$.\footnote{This type of functions was called \emph{pseudoconvex of order $n-q$} by O.~Fujita~\cite{Fu2}. Smooth \qpsh\ \fcts\ are exactly the \emph{weakly $(q+1)$-convex} ones in the sense of Grauert.}


\item If $m \geq n$, every \usc\ \fct\ on $U$ is by convention $m$-\psh.

\item $\PSH_q(U)$ donotes the family of \qpsh\ \fcts\ on $U$.

\item For $q=0$, we have the classical \textit{\psh}\ \fcts, and the $(n-1)$-\psh\ \fcts\ are also called \textit{subpluriharmonic}. 

\end{enumerate}

\end{defn}

The next property follows directly from the definition of \qpshy.

\begin{prop}\label{prop-sup-inf-qpsh} Let $\{\psi_j\}_{j \in J}$ be a collection of \qpsh\ \fcts\ on an open subset $U$ of $\cbb^n$. If $\{\psi_j\}_j$ is locally bounded, then the \usc\ regularization $\Psi^*$ of $\Psi:=\sup_{j \in J}\{\psi_j\}$ is \qpsh\ on $U$. If $\{\psi_j\}_{j \in \nbb}$ is a decreasing sequence, then $\inf_j\{\psi_j\}=\lim_j \psi_j$ is \qpsh\ on $U$.
\end{prop}

The maximum principle holds true for \qpsh\ \fcts~\cite{HM}.

\begin{thm}[Maximum principle] \label{prop-qpsh-locmax} Let $q\in\{ 0,\ldots,n-1\}$ and $U \relc \cbb^n$. Then any \fct\ $\psi \in \PSH_q(U) \cap \USC(\cl{U})$ fulfills
\[
\max\{ \psi(z) : z \in \cl{U}\} = \max\{\psi(z) : z \in \partial U\}.
\]
\end{thm}

The next result is Theorem~5.1 in~\cite{Sl84}.

\begin{thm}\label{thm-add-qpsh} Let $\psi \in \PSH_q(U)$ and $\varphi \in \PSH_r(U)$. Then $\psi+\varphi \in \PSH_{q+r}(U)$.
\end{thm}

The following important technique is from~\cite{Dieu}, but was also already used in some proofs in~\cite{HM} and~\cite{Sl84}.

\begin{lem}[Glueing lemma]\label{lem-glueing}

Let $W \subset U$, $w \in \PSH_q(W)$ and $u \in \PSH_q(U)$ such that $\limsup_{\zeta \to z} w(\zeta) \leq u(z)$ for every $z \in \partial W \cap U$. Then
\[
\psi(z):=\begin{cases} \max\{w(z),u(z)\}, & z \in \cl{W} \\ u(z), & z \in U\setminus W \end{cases}
\]
is \qpsh\ on $U$.
\end{lem}

We have the following characterization of smooth \qpsh\ \fcts.

\begin{thm}\label{smooth-qpsh} Let $q \in \{0,\ldots,n-1\}$ and let $\psi$ be a $\ccal^2$-smooth \fct\ on an open subset $U$ in $\cbb^n$. Then $\psi \in \PSH_q(U)$ if and only if the complex Hessian $\left( \frac{\partial^2 \psi}{\partial z_k \partial \overline{z}_l}(p)\right)_{k,l=1}^n$ has at most~$q$ negative eigenvalues at every point $p$ in~$U$.
\end{thm}

\begin{defn}\label{defn-sqpsh} We say that $\psi$ is \textit{strictly \qpsh}\ on $U$ if it is $\ccal^2$-smooth and its complex Hessian has at most $q$ non-positive eigenvalues.
\end{defn}

Using approximation techniques from \cite{Sl84} and~\cite{Bu}, we obtain the subsequent characterization of \qpshy\ (see Corollary 3.5.5 in~\cite{TPTHESIS}).

\begin{lem}\label{lem-char-qpsh} Let $q \in \{0,1,\ldots,n-1\}$ and $U$ be open in $\cbb^n$. Then $\psi \in \PSH_q(U)$ if and only if for every open subset $V \subseteq U$ and every $\varphi \in \PSH_{n-q-1}(V) \cap \ccal(V)$ the sum $\psi+\varphi$ is subpluriharmonic on $V$.
\end{lem}

The next notion is from~\cite{Sl84}. 

\begin{defn}\label{defqBrem} Let $q\in\{ 0,1,\ldots,n{-}1\}$ and $A$ be an arbitrary set in $\cbb^n$. Let $F:A\to\rbb$ be real-valued. We say that $F$ is \textit{almost \qbrem} on $A$, if $F^* \in \PSH_{q}(A^\circ)$ and $-F_* \in \PSH_{n-q-1}(A^\circ)$. Here, $A^\circ$ denotes the interior of $A$. The \fct\ $F$ is called \textit{\qbrem}\ on $A$ if $F$ is \cont\ on $A$ and almost \qbrem\ on~$A^\circ$.
\end{defn}

\begin{rem} Every \qbrem\ \fct\ is almost \qbrem, and every \cont\ almost \qbrem\ \fct\ is \qbrem, since $F^*=F$ and $F_*=F$ if $F$ is \cont. If $F$ is \qbrem\ and $\ccal^2$-smooth, then its complex Hessian has at least one zero eigenvalue. In this sense, \qbrem\ \fcts\ $F$ satisfy the complex Monge-Amp\`ere equation $(\partial\cl{\partial} F)^n=0$. By this reason, \qbrem\ \fcts\ were called \textit{$q$-complex Monge-Amp\`ere} or \textit{$q$-CMA} in~\cite{HM}. Therefore, every \ph\ \fct\ on an open subset $U$ of $\cbb^n$ is $q$-Bremermann on $U$ for every $q\in\{0,1,\ldots,n-1\}$. If $F$ is $0$-Bremermann, then $F^*$ is \psh\ and $-F_*$ is subpluriharmonic.
\end{rem}

\begin{ex} Let $\|z\|$ be any complex norm on $\cbb^n$. Then $\log\|z\|$ is \psh\ on $\cbb^n$ and $-\log\|z\|$ is subpluriharmonic on $\cbb^n\setminus\{0\}$ (see \cite{TPESZ}). This means that $\log\|z\|$ is $0$-Bremermann on $\cbb^n\setminus\{0\}$. As a consequece, $\log\|F\|$ is $0$-Bremermann on $\{F \neq 0\}$ for every holomorphic map $F:U \to \cbb^k$ and $U \subset \cbb^n$ open.
\end{ex}

\begin{prop}\label{prop-sup-inf-qbrem} Let $D$ be a domain in $\cbb^n$. If $\{F_j\}_j$ is a locally bounded increasing sequence of \qbrem\ \fcts\ on $\cld$, then $F:=\sup_j F_j$ is almost \qbrem\ on $\cld$. If $\{G_j\}_j$ is a locally bounded decreasing sequence of \qbrem\ \fcts\ on $\cld$, then $G:=\inf_j G_j$ is almost \qbrem\ on~$\cld$. 
\end{prop}

\begin{proof} By Proposition~\ref{prop-sup-inf-qpsh}, $F^* \in \PSH_q(D)$. Since $F_j$'s are \cont, $F$ is \lsc\ on $D$, so $-F$ is \usc. Moreover, $-F=\inf_j\{-F_j\}$, so $-F$ is the limit of a decreasing sequence of \fcts\ $-F_j \in \PSH_{n-q-1}(D)$. Again, by Proposition~\ref{prop-sup-inf-qpsh}, $-F=-F_*$ is also in $\PSH_{n-q-1}(D)$.

Now notice that, in general, a \fct\ $H$ is (almost) \qbrem\ if and only if $-H$ is (almost) $(n-q-1)$-Bremermann. Then $\{-G_j\}_j$ is an increasing sequence of $(n-q-1)$-Bremermann \fcts. Thus, by the first part, $-G=\sup_j\{-G_j\}$ is almost $(n-q-1)$-Bremermann on $D$, so that $G$ is $q$-Bremermann on~$D$.
\end{proof}

The next result can be found in \cite{Sl84}.

\begin{prop}[Uniqueness property]\label{prop-unique-qbrem} Let $D$ be a bounded domain in $\cbb^n$, and let $F,G$ be two almost \qbrem\ \fcts\ on $\cl{D}$ such that $F,G \in \ccal(\partial D)$ and $F=G$ on $\partial D$. Then $F=G$ on $\cld$.
\end{prop}

\begin{proof} By Thoerem \ref{thm-add-qpsh} and since $F=F_*=F^*$ and $G=G_*=G^*$ on $\partial D$, $F^*-G_*$ and $G^*-F_*$ are $(n-1)$-\psh\ on $D$ such that $F^*-G_*= 0$ and $G^*-F_*=0$ on $\partial D$. Since $F=G$ on $\partial D$, we can apply the maximum principle (Theorem~\ref{thm-add-qpsh}) in order to obtain $F\leq F^* \leq G_* \leq G$ and $G\leq G^* \leq F_* \leq F$ on~$\cld$.
\end{proof}

This motivates to study the maximal property of \qbrem\ \fcts\ in the subsequent sense.

\begin{defn} Let $D$ be a domain in $\cbb^n$ and $F:\cld \to \rbb$ a real-valued \fct. We say that $F$ is \textit{$q$-maximal on $\cld$} if $F^*\in\PSH_q(D)\cap\USC(\cld)$ and if for every $\psi \in \PSH_q(D)\cap\USC(\cl{D})$ with $\psi \leq F_*$ on $\partial D$ we have that $\psi \leq F_*$ on~$\cld$. We say that $F$ is \textit{$q$-maximal on (bounded) open subsets of $U$} if it is $q$-maximal on $\cl{U}$ for every (bounded) open subset $U$ of $D$.
\end{defn}

\begin{rem} If $F$ is $q$-maximal on open subsets of $D$, then it is $q$-maximal on $\cld$ and on bounded open subsets of $D$. If $D$ is bounded, $q$-maximality on open sets and $q$-maximality on bounded open sets are the same.
\end{rem}

\begin{prop}\label{prop-equiv-glob-loc-q-max} Let $D$ be a domain in $\cbb^n$ and $F\in\ccal(\cld)$. Then $F$ is $q$-maximal on $\cld$ if and only if it is $q$-maximal on open subsets of $D$.
\end{prop}

\begin{proof} Let $F$ be $q$-maximal on $\cld$ and let $U$ be an open subset of $D$. We pick a \fct\ $\psi \in \PSH_q(U)\cap \USC(\cl{U})$ with $\psi \leq F$ on $\partial U$. Then by the Glueing Lemma~\ref{lem-glueing},
\[
\Psi := \begin{cases} \max\{\psi,F\}, & \text{on}\ \cl{U}\\ F, & \text{on}\ \cld\setminus U \end{cases}
\]
is in $\PSH_q(U)\cap \USC(\cl{U})$ and fulfills $\Psi \leq F$ on $\partial D$. By the $q$-maximality of $F$, it follows that $\Psi\leq F$ on $\cld$, so that $\psi\leq\Psi \leq F$ on $U$.
\end{proof}

\begin{prop}\label{prop-a-q-brem-implies-loc-q-max} Let $D$ be a domain in $\cbb^n$ and let $F:\cld\to\rbb$ be an almost \qbrem\ \fct\ on $\cld$. Then $F$ is $q$-maximal on bounded open subsets of~$D$.
\end{prop}

\begin{proof} Let $U$ be a bounded open set in $D$ and $\psi \in \PSH_q(U)\cap\USC(\cl{U})$ with $\psi \leq F_*$ on $\partial U$. Then $\psi-F_*$ is $(n-1)$-\psh\ due to Theorem~\ref{thm-add-qpsh}, and $\psi-F_* \leq 0$ on $\partial U$. The maximum principle yields $\psi-F_*\leq 0$ on $U$. 
\end{proof}

\begin{prop}\label{prop-equiv-loc-q-max-2} Let $D$ be a domain in $\cbb^n$ and $F\in\ccal(\cld)$. Then $F$ is \qbrem\ on $\cld$ if and only $F$ is $q$-maximal on bounded open subsets of $D$.
\end{prop}

\begin{proof} Recall that $F^*=F=F_*$ if $F$ is \cont. Then the first direction follows immediately from Proposition~\ref{prop-a-q-brem-implies-loc-q-max}.

Now assume that $F$ is $q$-maximal on $\cld$. By the definition of $q$-maximality, we already have $F \in \PSH_q(D)$, so that it only remains to show that $-F \in \PSH_{n-q-1}(D)$. Assume that this is not the case. According to Lemma~\ref{lem-char-qpsh}, there exist a ball $B \relc D$ and a \cont\ \qpsh\ \fct\ $\varphi$ defined on an open \nbh\ of $\cl{B}$ such that $-F+\varphi<0$ on $\partial B$ and $-F(p)+\varphi(p)>0$ for some $p \in B$. This means that $\varphi < F$ on $\partial B$ and $\varphi(p)>F(p)$. Now define
\[
\Phi:=\begin{cases} \max\{F,\varphi\}, & \text{on}\ \cl{B}\\ F, & \text{on}\ \cld\setminus B \end{cases}
\]
By the Glueing Lemma~\ref{lem-glueing}, the \fct\ $\Phi \in \PSH_q(D) \cap \ccal(\cld)$ satisfying $\Phi = F$ on $\partial D$ and $\Phi(p)=\varphi(p)>F(p)$. This contradicts the $q$-maximality of $F$ on $\cld$. The proof is finished due to Proposition~\ref{prop-equiv-glob-loc-q-max}.
\end{proof}

\section{Properties of the q-Perron-Bremermann envelopes}

The solution of the Dirichlet-Bremermann problem for \psh\ \fcts\ uses the Perron-Bremermann envelope which easily extends to \qpsh\ \fcts. Notice that, in this section, $M \in \rbb\cup\{+\infty\}$, unless otherwise stated.

\begin{defn}\label{defn-q-Perron}  Let $D$ be a domain in $\cbb^n$, and $f \in \ccal(\partial D)$.  Fix a constant $M$ with $\sup_{\partial D} f \leq M \leq +\infty$. Then the \textit{$q$-Perron-Bremermann envelope} on $\cld$ is defined at each point $z$ of $\cld$ by
\[
\pcal_{f,q,D,M}(z):=\sup\{\psi(z) : \ \psi \in \PSH_q(D)\cap \ccal(\cld), \ \psi \leq f \ \text{on}\ \partial D, \ \psi \leq M \ \text{on}\ \cld\}.
\]
If $M=+\infty$, we simply write $\pcal_{f,q,D}$ rather than $\pcal_{f,q,D,+\infty}$ because in this case, the condition $\psi \leq M$ is superfluous.
\end{defn}

\begin{rem} It is obvious that $\pcal_{f,q,D,M} \leq M$ and $\sup_{M<+\infty} \pcal_{f,q,D,M} \leq \pcal_{f,q,D}$. Notice that $\pcal_{f,q,D,M} \leq f$ on $\partial D$, but the \usc\ regularization $\pcal_{f,q,D,M}^*$ might jump at the boundary and exceed $f$ there. This means that we cannot conclude immediately that $\pcal_{f,q,D,M}^*=\pcal_{f,q,D,M}$. If the latter identity holds true, $\pcal_{f,q,D,M}$ is \cont\ on $\cld$.

If $D$ is bounded, by the maximum principle, the assumption $\psi \leq f \leq M$ on $\partial D$ implies $\psi \leq M$ on~$\cld$. In this case, for any $M \geq \sup_{\partial D} f$, we have, $\pcal_{f,q,D,M}  = \pcal_{f,q,D} \leq \sup_{\partial D} f$. 

Since $\pcal_{f,q,D,M}$ is \lsc, it is $q$-maximal with respect to $\psi \in \PSH_q(D) \cap \ccal(\cld)$ such that $\psi \leq f$ on $\partial D$ and $\psi \leq M$ on $\cld$.
\end{rem}

\begin{rem}\label{rem-slightly} In the literature, the definition of the $q$-Perron-Bremermann envelope ($q=0$) is slightly different. E.g., in \cite{NH} or \cite{NTT} (for $q=0$) the following is used:
\[
\widetilde{\pcal}_{f,q,D}(z):=\sup\{\psi(z) : \ \psi \in \PSH(D), \ \psi^{\star} \leq f \ \text{on}\ \partial D\},
\]
where $\psi^{\star}(z):=\limsup_{\zeta \to z}\psi(\zeta)$ for $z \in \cld$. It is clear that it coincides with the definition in \cite{HM},
\[
\sup\{\psi(z) : \ \psi \in \PSH(D)\cap\USC(\cld), \ \psi \leq f \ \text{on}\ \partial D\}.
\]
The reason why we chose our definition is, that in contrary to $\pcal_{f,q,D}$, the \fct\ $\widetilde{\pcal}_{f,q,D}$ need not to be \lsc\ on $D$, which we use in several arguments. In general, we have $\pcal_{f,q,D} \leq \widetilde{\pcal}_{f,q,D}$. We will discuss conditions on their equality later.
\end{rem}

\begin{prop}\label{prop-q-envelope} Let $D$ be a domain in $\cbb^n$. The $q$-Perron-Bremermann envelope $\pcal:=\pcal_{f,q,D,M}$ is \lsc\ with $\pcal \leq f$ on $\partial D$, and it is almost \qbrem\ outside $E_{\infty}:=\{z \in \cld : \pcal(z)=+\infty\}$ in $\cld$. Moreover, it is \cont\ and \qbrem\ outside $E$ in $\cld$, where $E:=E_\infty \cup E_d$ and $E_d:=\{z \in \cld : \pcal(z) < \pcal^*(z)\}$. Hence, $\pcal$ is $q$-maximal on bounded open subsets of $D\setminus E$. Since $f \in \ccal(\partial D)$, we have $E \cap \partial D = \emptyset$. If $M<+\infty$ or if $D$ is bounded, then $E_\infty=\emptyset$.
\end{prop}

\begin{proof} $\pcal$ is the supremum of \cont\ \fcts\ on $\cld$, so it is \lsc, i.e., $\pcal_*=\pcal$, and $\pcal^* \in \PSH_q((D\setminus E_\infty)^\circ)$. It remains to show that $-\pcal$ is $(n-q-1)$-\psh\ on $(D\setminus E_\infty)^\circ$. Assume that this is not the case. Then by Lemma~\ref{lem-char-qpsh} there exist a ball $B \relc D\setminus E_\infty$ and a \fct\ $\varphi \in \PSH_q(B) \cap \ccal(\cl{B})$ such that $-\pcal + \varphi < 0$ on $\partial B$, but $-\pcal(p) + \varphi(p) >0$ for some point $p \in B$. In other words, $\vphi<\pcal\leq M$ on $\partial B$ and $\pcal(p) < \varphi(p)$. By the maximum principle, $\varphi(p)\leq M$. By the continuity of $\varphi$, the lower semi-continuity of $\pcal$ and the definition of the $q$-Perron-Bremermann envelope, there exists a $\psi \in \PSH_q(D) \cap \ccal(\cld)$ such that $\pcal \geq \psi > \varphi$ on $\partial B$,  $\psi(p)\leq\pcal(p) < \varphi(p)$ and $\psi \leq f$ on $\partial D$. Now define
\[
\Psi:= \begin{cases} \max\{\psi, \varphi\}, & \text{on}\ \cl{B}\\ \psi, & \text{on}\ \cld\setminus B \end{cases}.
\]
By the Glueing Lemma~\ref{lem-glueing}, $\Psi \in \PSH_q(D) \cap \ccal(\cld)$, $\Psi=\psi \leq f$ on $\partial D$, $\Psi(p)= \varphi(p) > \pcal(p)$ and $\Psi \leq M$ on~$\cld$. By the definition of the $q$-Perron-Bremermann envelope, we obtain $\Psi \leq \pcal$ on $\cld$. In particular, $\Psi(p) \leq \pcal(p)$, which is a contradiction. Hence, $-\pcal$ is $(n-q-1)$-\psh\ on $(D\setminus E_\infty)^\circ$. The rest is obvious, since $\pcal=\pcal_*$ and $E_d$ is the set of discontinuity of $\pcal$ on $\cld$.
\end{proof}

\begin{ex} Let $n \in \nbb$, $z \in \cbb$, and define
\[
s_n(z):=n(|\exp(\exp(z))|-1)=n(e^{e^x\cos(y)}-1),
\]
on $D:=\rbb\times(-\pi/2,\pi/2) \subset \cbb$. Then $s_n=0$ on $\partial D$ and $s_n>0$ on $D$. Let $f:=0$ on $\partial D$. Then, for every $n \in \nbb$, we have $s_n \leq \pcal_{f,q,D}$. Therefore, $\pcal_{f,q,D}=+\infty$ on $D$, and $E_\infty=D$. By the Riemann mapping theorem and the maximum principle, $\pcal_{f,q,D,M}=0$ on $\cld$ for any $M \geq 0$. Hence, $\sup_{M<+\infty}\pcal_{f,q,D,M} < \pcal_{f,q,D}$ on $D$. 
\end{ex}

\begin{prop}\label{prop-q-Perron-q-max} Let $D$ be a domain in $\cbb^n$.

\begin{enumerate}

\item If there exists a \fct\ $F$ on $\cld$ such that $F \in \ccal(\cld)$ and $F$ is $q$-maximal with respect to $\psi \in \PSH_q(D)\cap\USC(\cld)$ with $\psi \leq f$ on $\partial D$ and $\psi \leq M$ on~$\cld$, where $f:=F|_{\partial D}$ and $M\geq\sup_{\partial D} f$, then $\pcal_{f,q,D,M} = F$ is \cont\ on $\cld$.

\item $\pcal_{f,q,D,M} \in \ccal(\cld)$ is $q$-maximal on bounded open subsets of $D$. 
\end{enumerate}

\end{prop}

\begin{proof} (1) Obviously, $F\leq \pcal_{f,q,D,M}$ on $\cld$. On the other hand, for every $\psi$ as above in (1), by $q$-maximality of $F$, we have $\psi \leq F$ on $\cld$. Thus, $\pcal_{f,q,D,M} \leq F$ on~$\cld$, and therefore, $\pcal_{f,q,D,M} = F$ on $\cld$.

(2) Since $\pcal_{f,q,D,M} \in \ccal(\cld)$, it is \qbrem\ on $\cld$, and therefore $q$-maximal on bounded open subsets of $D$ due to Proposition~\ref{prop-equiv-loc-q-max-2}.
\end{proof}

\section{Solution of the q-Dirichlet-Bremermann problem}

The $q$-Perron-Bremermann envelope attains the prescribed values at the boundary when the domain permits peak points at its boundary.

\begin{defn}\label{defn-q-peak} An open set $U$ in $\cbb^n$ has a \textit{$q$-peak point} at a boundary point $p \in \partial U$ if there is a \textit{$q$-peak \fct}\ $\psi \in \PSH_q(U) \cap \ccal(\cl{U})$ such that $\psi(p)=0$ and $\psi<0$ on $\cl{U}\setminus\{p\}$.
\end{defn}

\begin{rem} Every \textit{local} $q$-peak point for $U$ is a $q$-peak point for $U$. Indeed, assume that there is an open \nbh\ $V$ of $p$ such that $U \cap V$ has a $q$-peak point at $p$. Take an open ball $B \relc V$ centered at $p$ and let $C:=\max_{\partial B} \psi$. Then $\varphi:=\max\{\psi,a\}$ extends by the constant $C<a<0$ from $\cl{U} \cap \cl{B}$ to a \qpsh\ \fct\ on the whole of $U$ that is \usc\ on $\cl{U}$ and peaks at~$p$.
\end{rem}

The next lemma clarifies the relevance of $q$-peak points for the $q$-Perron-Bremermann envelope.

\begin{lem}\label{lem-q-peak} Let $p \in \partial D$ be a $q$-peak point for the domain $D$ in $\cbb^n$ with $q$-peak function $\psi$, and let $f \in \ccal(\partial D)$ with $f \geq0$. Then $\pcal_{f,q,D,M}(p) = f(p)$ for every $M \geq \sup_{\partial D} f$.
\end{lem}

\begin{proof} For every $\eps>0$ there is a radius $r>0$ such that $\Psi_{a,\eps}:=a\psi+f(p)-\eps < f$ on the closure of the ball $B=B_r(p)$ in $\partial D$ for all $a\geq1$. If $a\geq1$ is large enough, $\Psi_{a,\eps} < 0$ on $\partial B \cap \cld$. By the Glueing Lemma~\ref{lem-glueing}, $\widetilde{\Psi}_{a,\eps}:=\max\{0,\Psi_{a,\eps}\} \in \PSH_q(D)\cap\ccal(\cld)$ such that $\widetilde{\Psi}_{a,\eps} \leq f$ on $\partial D$ and $\widetilde{\Psi}_{a,\eps} \leq \sup_{\partial D} f \leq M$ on $\cld$. Therefore, $\widetilde{\Psi}_{a,\eps}(p)=f(p)-\eps \leq \pcal_{f,q,D,M}(p) \leq f(p)$. Since $\eps>0$ was chosen arbitrarily, we conclude $\pcal_{f,q,D,M}(p) = f(p)$.
\end{proof}

Domains that admit peak points at its boundary are given by the following important class of sets.

\begin{defn}\label{defn-sqpsc} We say that a domain $D$ in $\cbb^n$ is \textit{strictly (Levi) \qpsc}\ at $p \in \partial D$ if there is a $\ccal^2$-smooth \sqpsh\ \fct\ $\psi$ on an open \nbh\ $U$ of $p$ such that $D \cap U=\{z \in U : \psi<0\}$ and $d\psi(p)\neq 0$.
\end{defn}

\begin{rem}\label{rem-q-Dirichlet} The classical \spscy\ is strict $0$-pseudoconvexity in our sense. This means that the result in Theorem \ref{thm-q-Dirichlet} always holds true in the case of $r=0$ for any $q \in \{0,1,\ldots,n-1\}$. If $D$ is strictly Levi \qpsc\ at $p \in \partial D$, then it has a (local) $q$-peak point at $p$. In view of Definition~\ref{defn-sqpsc} we simply can take $\varphi:=\psi-\eps|z-p|^2$ as a local peak function for $D$ at $p$, for $\eps>0$ small enough.
\end{rem}

The following theorem is due to \cite{HM} in the case $r=q$ and in its general version to~\cite{Ka}. Uniqueness of the solution has been shown in~\cite{Sl84}. Notice also the example and the comment before Theorem 3.7 in~\cite{Bu} on the extra condition $r \leq q \leq n-r-1$. In~\cite{Dieu}, it was pointed out that $r$-peak points (or $B_q$-regular points) at the boundary of the given domain  are sufficient to get a solution of the $q$-Dirichlet-Bremermann problem.

\begin{thm}\label{thm-q-Dirichlet} Let $r \in \{0,1,\ldots,n-1\}$ and let $D$ be a bounded domain in $\cbb^n$ that admits an $r$-peak point at each of its boundary points. Then for $r \leq q \leq n-r-1$ and $f \in \ccal(\partial D)$, the $q$-Perron-Bremermann envelope $\widetilde{\pcal}_{f,q,D}$ is \cont\ on $\cld$. Moreover, $\widetilde{\pcal}_{f,q,D}$ is the unique $q$-Bremermann \fct\ on $D$ such that $\D \widetilde{\pcal}_{f,q,D} = f$ on $\partial D$. Therefore, it is $q$-maximal on open subsets of $D$.
\end{thm}


\begin{rem} It follows from the definition that $\pcal_{f,q,D} \leq \widetilde{\pcal}_{f,q,D}$. Since $\widetilde{\pcal}_{f,q,D}$ is \cont\ on $\cld$ and $\widetilde{\pcal}_{f,q,D}=f$ on $\partial D$, we also obtain $\widetilde{\pcal}_{f,q,D} \leq \pcal_{f,q,D}$. Therefore, $\pcal_{f,q,D} = \widetilde{\pcal}_{f,q,D}$ under the assumptions made in Theorem~\ref{thm-q-Dirichlet}.
\end{rem}

\begin{cor}\label{prop-equiv-loc-q-max} If $D$ is a domain in $\cbb^n$ and $F \in \ccal(\cld)$. $F$ is \qbrem\ on $\cld$ if and only if $F=\pcal_{F|_{\partial B},q,B}$ for every ball $B \subseteq D$. 
\end{cor}

\begin{proof}
One direction follows from the uniqueness of \qbrem\ \fcts\ (Proposition~\ref{prop-unique-qbrem}). If $F=\pcal_{F|_{\partial B},q,B}$ on $B$ for every ball $B$ in $D$, then $F$ is locally \qbrem. Since \qpshy\ is a local property, the proof is finished.
\end{proof}

\begin{rem} The boundedness of $D$ guarantees that $\pcal_{f,q,D} \leq \max_{\partial D} f < +\infty$ on $\cld$ by the maximum principle, and Walsh's technique \cite{Wa} can then be applied to establish the continuity of $\pcal_{f,q,D}$ on $\cld$. For an unbounded domain $D$, however, it is not immediate that $\pcal_{f,q,D}$ is bounded, and even when boundedness does hold, continuity remains the most delicate issue. As we shall see later, this issue plays a central role in our main result, Theorem~\ref{thm-q-Dirichlet-unbounded}. 
\end{rem}

Consider also the following Example 8.2 from~\cite{ShTo} on the boundedness of the boundary data.

\begin{ex}\label{ex-Sh-To} Consider the domain $D=\{(z,u+iv) \in \cbb^2 : v > |z|^2 + u^2\}$. It is strictly convex in $\cbb^2$, and there exists a \fct\ $f \in \ccal(\partial D)$ with $f \leq 0$ such that for each $\psi \in \PSH(\cld)$ with $\psi \leq f$ on $\partial D$ we have $\psi\equiv-\infty$ on $D$. Therefore, $\pcal_{f,{q=0},D}=-\infty$ on $D$.
\end{ex}

In virtue of the previous example, from now on, we only consider continuous boundary values that are bounded below. The next result is Theorem~24 and Theorem~25 in~\cite{SiTo}.

\begin{thm}\label{thm-SiTo} Let $D$ be an unbounded strictly convex domain in $\cbb^n$ and $f \in \ccal(\partial D)$ with $f\geq0$. Then there exists a \fct\ $\Phi \in \PSH_q(D) \cap \USC(\cld)$ such that $\Phi \geq0$, $\Phi|_{\partial D}=f$ and $\Phi \in \ccal(\partial D)$. Moreover, $\Phi$ is $q$-maximal on $\cld$. If $f$ is bounded, $\Phi$ is \cont\ and \qbrem\ on $\cld$.
\end{thm}

\begin{rem} It follows from Proposition~\ref{prop-q-Perron-q-max} that $\Phi=\pcal_{f,q,D}$ on $\cld$, and from Proposition~\ref{prop-q-envelope} that $\Phi$ is almost $q$-Bremermann on $\cld$.
\end{rem}

Now we construct a \cont\ \qpsh\ extension out of given \cont\ boundary data (see also \cite{HST}).

\begin{lem}\label{lem-qpsh-extension} Let $r \in \{0,1,\ldots,n-1\}$ and let $D$ be a domain in $\cbb^n$ that admits an $r$-peak point at each of its boundary points. Fix $r \leq q \leq n-r-1$. Then, for every $f\in\ccal(\partial D)$ with $f \geq 0$, there exists $F \in \PSH_q(D) \cap \ccal(\cld)$ such that $F|_{\partial D}=f$, $F\geq0$ and $\sup_{\cld} F = \sup_{\partial D} f$.
\end{lem}

\begin{proof} Let $\{B_i\}_i$, $\{B_i'\}_i$ and $\{B_i''\}_i$ be locally finite coverings of $\partial D$ by balls centered in $p_i \in \partial D$ with $B_i \relc B_i' \relc B_i''$. Let $D_i:=B''_i \cap D$. For each $i$, we pick a \cont\ function $f_i$ on $\overline{D_i}$ such that $f_i=f$ on $\partial D\cap\overline{B}_i$,  $f_i < 0$ on $\partial D_i \setminus \overline{B_i'}$, $f_i\leq f$ on $\partial D \cap B_i''$. Then $\sup_{\overline{D_i}} f_i = \sup_{\overline{D_i}} f$. Since all boundary points of $D_i$ are $r$-peak points, it follows from  Theorem~\ref{thm-q-Dirichlet} that for each $f_i$ there is a \qbrem\ \fct\ $F_i$ on $\overline{D_i}$ such that $F_i=f_i$ on $\partial D_i$ and $\sup_{\overline{D_i}} F_i = \sup_{\overline{D_i}} f$. We set $F_i':=\max\{ F_i, 0\}$. Then $F_i'$ is \qpsh\ on $D_i$ and \cont\ on $\cl{D_i}$, and $F_i'<0$ near $\partial D_i\setminus B_i'$. Therefore, the function $F_i'':\cld \to \rbb$ given by $F_i'':=F_i'$ on $\overline{D_i}$ and $F_i'':=0$ on $\cld\setminus B_i''$ is \cont\ on $\cld$ and \qpsh\ on $D$. Since $\{B_i\}_i$ is a locally finite covering of $\partial D$, the function $F:=\sup_i\{F_i''\}=\max_i\{ F_i''\}$ is \qpsh\ on $D$ such that $F=f_i=f$ on $\partial D \cap B_i$. It is evident that
\[
\sup_{\cld} F \leq \sup_i \sup_{\overline{D_i}} F_i'' \leq \sup_{i}\sup_{\overline{D_i}} F_i \leq \sup_{\cl{D}} f.
\]
\end{proof}

\begin{rem} Observe that $F=0$ on $D\setminus \bigcup_i B_i''$. Since the covering $\{B_i''\}_i$ can be chosen arbitrarily small, we can construct $F$ to vanish on any closed subset $A \subset D$ with $A \cap \partial D = \emptyset$.
\end{rem}


The next notion can be found in \cite{NH}.

\begin{defn}\label{defn-bounded-type} An unbounded domain $D$ in $\cbb^n$ is of \textit{bounded type} if there exist $\Psi \in \PSH(D)\cap \ccal(D)$ such that $\Psi <0$ on $D$ and $\Psi(z)\to-\infty$ as $|z|\to+\infty$, i.e., for every $C\geq0$ there exists $R>0$ such that $\Psi < -C$ on $D\setminus B_R(0)$.
\end{defn}

\begin{ex} If a domain $D$ contains a copy of $\cbb$, then it is not of bounded type by the Liouville property for \psh\ \fcts. For example, the subsequent domain $D$ is \spsc, but not of bounded type, since it contains  $\cbb\times\{0\}$,
\[
\{(z,w) \in \cbb^2 : \log|w| + |z|^2+|w|^2<1\}.
\]
If a domain $D$ is contained in $\{z \in \cbb^n : |p(z)|^2>(1+|z|^2)^{\deg p}\}$ for a complex polynomial~$p \in \cbb[z_1,\ldots,z_n]$, then $D$ is of bounded type (see \cite{NH}). 
\end{ex}

The maximum principle holds true on unbounded domains of bounded type.

\begin{prop} \label{prop-qpsh-max-princ-2} Let $q\in\{ 0,\ldots,n-1\}$ and let $D$ be a domain in $\cbb^n$ of bounded type. Then any bounded above \fct\ $\vphi \in \PSH_q(D) \cap \USC(\cl{D})$ satisfies
\[
\sup\{ \vphi(z) : z \in \cl{D}\} = \sup\{\vphi(z) : z \in \partial D\}.
\]
\end{prop}

\begin{proof} Assume that the assumption is false, and that there exists $\vphi \in \PSH_q(D) \cap \USC(\cl{D})$ such that $
m:=\sup_{\partial D}\vphi < M:=\sup_{\cld} \vphi < +\infty$. Let $\Psi$ be from Definition~\ref{defn-bounded-type}, and $\widehat{\Psi}$ its \cont\ extension to $\cld$. We fix $\eps>0$ and define $\Phi_{\eps}:=\varphi+\eps\widehat{\Psi}$. Let $R_0>0$ be so large that $\eps\widehat{\Psi}< m-M$ on $\cld\setminus B_{R_0}(0)$ for every $R\geq R_0$. Then $\Phi_{\eps} \leq m$ on $\cld\setminus B_{R}(0)$, and by the maxmum principle on bounded domains, $\Phi_{\eps} \leq m$ on $\cld \cap \cl{B_{R}(0)}$ for every $R \geq R_0$. Therefore, $\Phi_{\eps} \leq m$ on $\cld$. Since $\eps>0$ was arbitrarily chosen, $\varphi \leq m$ on $\cld$. But this contradicts to $M\leq\sup_{\cld} \vphi \leq m$.
\end{proof}

\begin{rem}\label{rem-qpsh-max-princ-2} If $D$ is of bounded type and $f \in\ccal(\partial D)$ is bounded above, then $\pcal_{f,q,D,M} =\pcal_{f,q,D,M_0}$ for every $M$ with $M_0:=\sup_{\partial D} f \leq M < +\infty$.
\end{rem}

We present our main result on the solution of the $q$-Dirichlet-Bremermann problem on unbounded domains of bounded type.

\begin{thm}\label{thm-q-Dirichlet-unbounded}
Let $r$, $q$, $D$ be as in Theorem~\ref{thm-q-Dirichlet}, except that $D$ is unbounded. Let $f \in \ccal(\partial D)$ with $0 \leq f \leq \sup_{\partial D} f \leq M$ for some upper bound $M < +\infty$. Then there exists an almost \qbrem\ \fct\ $\Phi=\Phi_M$ on $\cld$ such that

\begin{enumerate}

\item $\Phi=f$ on $\partial D$, and $0 \leq \Phi \leq M$ on $\cld$.

\item $\Phi$ is $q$-maximal with respect to \qpsh\ \fcts\ $\psi$ with $\psi \leq f$ on $\partial D$ and $\psi \leq M$ on $\cld$. Therefore, $\pcal_{f,q,D,M} \leq \Phi$.

\item $\Phi$ is \cont\ on $\partial D$.

\item If $D$ is of bounded type, $\Phi$ is \cont\ and \qbrem\ on $\cld$. Thus, $\Phi=\pcal_{f,q,D,M_0}$ for $M_0:=\sup_{\partial D} f$.

\item If $D$ is of bounded type, $\Phi$ is the unique \qbrem\ \fct\ on $\cld$ with $\Phi=f$ on $\partial D$.

\end{enumerate}

\end{thm}

\begin{proof} 

\textit{Claim 1:} Existence of an \usc\ almost \qbrem\ $\Phi$ with $\Phi=f$ on $\partial D$ and $0 \leq \Phi \leq M$.

\medskip

For $j \in \nbb_0$ we define $D_j:=D\cap B_{j}(0)$,  $E_j:=\partial B_{j}(0)\cap \cld$ and $A_j:=\partial D \cap \cl{B_{j}(0)}$. We may assume that $j$ is so large that $D_0$ is not empty.

For $j \geq 1$, let $\vphi_{j} \in \ccal(\cl{D_j})$ such that $ 0 \leq f \leq \vphi_{j}\leq M$ on $\partial D_j \cap \partial D$, $\vphi_{j}=f$ on $\partial D_{j-1}\cap \partial D$ and $\vphi_j=M$ on $E_j$. Let $\Phi_{j}$ be the solution of the $q$-Dirichlet-Bremermann problem with boundary data $\vphi_{j}$ on $D_j$. Then by Theorem~\ref{thm-q-Dirichlet}, $\Phi_j=\pcal_{\vphi_j,q,D_j}$ is \qbrem\ on $\cl{D_j}$, $\Phi_j=\vphi_j$ on $\partial D_j$, $q$-maximal on open subsets of $D_j$, and fulfills $0 \leq \Phi_j \leq M$ on $\cl{D_j}$. 

Since $\Phi_{j+1}=f$ on $\partial D_j\cap \partial D$ and $\Phi_{j+1} \leq M$ on $D_j$, we have that $\Phi_{j+1} \leq \vphi_{j}$ on $\partial D_j$. Since $\Phi_{j+1}$ is \qpsh\ on $D_j$ and $\Phi_{j}$ is $q$-maximal on $\cl{D_j}$, we have that $\Phi_{j+1} \leq \Phi_{j}$ on $\cl{D_j}$. Hence, for each $j_0$, we obtain a locally bounded, decreasing sequence of \qbrem\ \fcts\ $\{\Phi_j\}_{j\geq j_0}$ on $\cl{D_{j_0}}$ such that $\Phi_j=f$ on $\partial D \cap \partial D_{j-1}$. Since $\{D_j\}_{j}$ is an increasing sequence of domains such that $D=\bigcup_{j}D_j$, the \fct\ $\Phi:=\inf_j \Phi_j$ is well-defined on the whole of $\cld$. Furthermore, it fulfills $\Phi=f$ on $\partial D$ and $0\leq\Phi \leq M$ on $\cl{D_j}$ for each $j$, so $\Phi \leq M$ on $\cld$. By Proposition~\ref{prop-sup-inf-qpsh}, $\Phi$ is almost \qbrem\ on $\cld$, and by the continuity of $\Phi_j$'s, $\Phi \in \USC(\cld)$.

\medskip

\textit{Claim 2:} $\Phi$ is $q$-maximal with respect to \qpsh\ \fcts\ $\psi$ with $\psi \leq f$ on $\partial D$ and $\psi \leq M$ on $\cld$, and $\pcal_{f,q,D,M} \leq \Phi$.

\medskip

Let $\psi \in \PSH_q(D)\cap\USC(\cld)$ such that $\psi \leq f$ on $\partial D$ and $\psi \leq M$ on $D$. Define $\psi_j:=\psi|_{\cl{D_j}}$. Then, by the construction of $\vphi_j$ and $\Phi_j$, we have $\psi_j\leq \vphi_j=\Phi_j$ on $\partial D_j$ for each $j$. Since $\Phi_j$ is $q$-Bremermann on $\cl{D_j}$, we conclude that $\psi=\psi_j \leq \Phi_j$ on $\cl{D_j}$ for every $j$, so that $\psi \leq \Phi$ on $\cl{D}$. Therefore, $\pcal_{f,q,D,M} \leq \Phi$ on~$\cld$.

\medskip

\textit{Claim 3:} $\Phi$ is \cont\ on $\partial D$.

\medskip

We only need to show that $\Phi$ is \lsc\ on $\partial D$, since it is already \usc\ on $\cld$. By Lemma~\ref{lem-qpsh-extension}, there exists a \fct\ $F \in \PSH_q(D)\cap\ccal(\cld)$ such that $F=f$ on $\partial D$ and $0 \leq F \leq \sup_{\cld} F = \sup_{\partial D} f \leq M$. By $q$-maximality of $\Phi$, we have $F\leq\Phi$ on $\cld$. Thus, if $p \in \partial D$, then
\[
\liminf_{z \to p} \Phi(z) \geq \liminf_{z \to p} F(z) = \lim_{z \to p} F(z) = F(p)=f(p)=\Phi(p)
\]




\medskip

\textit{Claim 4:} If $D$ is of bounded type, $\Phi$ is \cont\ and \qbrem\ on $\cl D$, and $\Phi \leq\pcal_{f,q,D,M}$.

\medskip

We use similar arguments as in the proof of Lemma~2 in~\cite{SiTo}. It suffices to show that $\Phi$ is \lsc\ on $D$. Since $D$ is of bounded type, there is a \cont\ negative \psh\ function $\Psi$ on $D$ such that $\Psi(z) \to -\infty$ as $|z|\to+\infty$.

We fix a point $p_0 \in D$, an upper bound $M \in [\sup_{\partial D} f,+\infty)$ and an arbitrary $\eps>0$. We choose a constant $a>0$ so small that $-\eps/2 < a\Psi(p_0)$. We replace $\Psi$ by $a\Psi$, so that we can assume $-\eps/2 < \Psi(p_0)$. Since $\Psi$ is \cont\ on $D$, we obtain that $-\eps/2 < \Psi$ on a ball $B_{\rho}(p_0) \relc D$ for some $\rho>0$.

Let $r_2>r_1>0$ be so large that $B_{\rho}(p_0) \relc B_{r_1}(0)$ and $\Psi +M < 0$ on $\cld \setminus B_{r_1}(0)$. We set $B':=B_{r_1}(0)$, $B'':=B_{r_2}(0)$ and $D'':=D\setminus \cl{B''}$.
Since $\Phi$ is \cont\ on $\partial D$, it is uniformly \cont\ on the compact set $\partial D \cap \cl{B''}$. Hence, there is a $\delta''>0$ such that for every $p \in \partial D \cap \cl{B''}$ and $z \in \cld \cap \cl{B''}$ with $|p-z|<2\delta''$ we have
\[
|\Phi(p) - \Phi(z)| < \eps/4.
\]
Let $D_1:=\{z \in D \cap B'' : d(z,\partial D)) \geq \delta/3\}$ and $D_2:=\{z \in D \cap B'' : d(z,\partial D)) < \delta/3\}$.

We fix a positive constant $\delta< \min\{\rho,r_2-r_1,\delta''\}$. For $w \in \cbb^n$ with $|w|<\delta/3$ we define
\[
\widetilde{\Phi}(z):= \begin{cases} \max\{\Phi(z), \Phi(z+w)+\Psi(z)-\eps/2\},  & z \in D_1  \\ \Phi(z) & z \in D_2 \cup (\cld \setminus B')  \end{cases}
\]
It is clear that, if $z \in D_1$, then $z+w \in \cld$. Now let $z \in D \cap B'$ such that $d(z,\partial D)=\delta/3$. Then there exists a $p \in \partial D$ such that $z \in B_\delta(p)$. Then $z+w \in B_{2\delta}(p)$, so that $|\Phi(p)-\Phi(z)|<\eps/2$ and $|\Phi(p)-\Phi(z+w)|<\eps/2$ due to uniform continuity of $\Phi$ on $\partial D$ discussed above. Therefore, since $\Psi<0$ on $D$,
\[
\Phi(z)> \Phi(z+w) - \eps/2 > \Phi(z+w) + \Psi(z) - \eps/2.
\]
Thus, $\widetilde{\Phi}$ is \qpsh\ on $D \cap B''$ by the Glueing Lemma~\ref{lem-glueing}.

If $z \in \partial B'' \cap \cld$ with $d(z,\partial D)\geq\delta/3$, then $z+w \in \cld \setminus B'$ and therefore, 
\[
\Phi(z+w)+\Psi(z)-\eps/2 \leq M + \Psi(z) < 0 \leq \Phi(z),
\]
so $\widetilde{\Phi}$ is \qpsh\ on $D$ and \usc\ on $\cld$ again by the Glueing Lemma~\ref{lem-glueing}. It is easy to see that $0 \leq \widetilde{\Phi} \leq M$ and $\widetilde{\Phi}=\Phi=f$ on $\partial D$. By the $q$-maximality of $\Phi$ on $\cld$ (Claim 2), we have $\widetilde{\Phi} \leq \Phi$ on $\cld$. Recall that we fixed a point $p_0 \in D$ such that $B_{\rho}(p_0) \relc B' \cap D$,   and $-\eps/2 < \Psi$ on $B_{\rho}(p_0)$. Hence, $d(p_0,\partial D) > \rho>\delta$. Then $z=p_0-w \in D_1$ for $|w| < \delta/3$, so that
\[
\Phi(p_0) -\eps =\Phi(z+w) -\eps < \Phi(z+w) + \Psi(z) -\eps/2 \leq \widetilde{\Phi}(z) \leq \Phi(z).  
\]
This means that $\Phi$ is \lsc\ at $p_0$. Since $p_0$ was an arbitrary point in $D$, we conclude that $\Phi$ is \lsc\ on the whole of $\cl{D}$. Since we already deduced that $\Phi$ is \usc\ on $\cld$, we finally get that $\Phi$ is \cont\ on the whole of $\cld$. Since $\Phi$ is almost \qbrem\ on $\cld$, it is \qbrem\ on $\cld$. Moreover, also by the continuity and the properties of $\Phi$, we have $\Phi \leq \pcal_{f,q,D,M}$. It follows from the maximum principle of domains of bounded type that $\Phi=\pcal_{f,q,D,M}=\pcal_{f,q,D,M_0}$ for every $M \geq M_0:=\sup_{\partial D} f$ (see Remark~\ref{rem-qpsh-max-princ-2}).

\medskip

\textit{Claim 5:} $\Phi=\pcal_{f,q,D,M_0}$ is the unique \qbrem\ \fct\ on $\cld$ that equals $f$ on~$\partial D$.

\medskip

Assume that another \qbrem\ $F$ exists on $\cld$ with $F=f$ on $\partial D$. It is clear that $F \leq \Phi$ on $\cld$ by the maximum principle on domains of bounded type and the definition of $\pcal_{f,q,D,M_0}$. Recall that $F$ is $q$-maximal on bounded open subsets of $D$, since it is \qbrem\ on $\cld$ (see Proposition~\ref{prop-equiv-loc-q-max-2}).

For arbitrary $\eps>0$ define $\Phi_\eps:=\Phi+\eps \widehat{\Psi}$ (where $\Psi$ is from Claim 4 and $\widehat{\Psi}$ is the \cont\ extension of $\Psi$ into the whole of $\cld$). If the ball $B=B_{r_0}(0)$ is large enough ($r_0>1/\eps$), then $\Phi_\eps \leq F$ on $\cld \setminus B$, so $\Phi_\eps \leq F$ on $\partial(D \cap B_r(0))$ for every $r \geq r_0$. By the $q$-maximality of $F$ on bounded open sets, we deduce $\Phi_\eps \leq F$ on $\cld$. Since $\eps>0$ was chosen to be arbitrary, we obtain $\Phi \leq F$ on $\cld$.
\end{proof}

\begin{rem} In view of Remark~\ref{rem-slightly}, for $M \geq \sup_{\partial D}f$ and $z \in D$, we define 
\[
\widetilde{\pcal}_{f,q,D,M}(z):=\sup\{\psi(z) : \ \psi \in \PSH_q(D)\cap \USC(\cld), \ \psi \leq f \ \text{on}\ \partial D, \ \psi \leq M \ \text{on}\ \cld\}.
\]
Then we can replace $\pcal_{f,q,D,M}$ by $\widetilde{\pcal}_{f,q,D,M}$ in Claim~2 in order to get $\widetilde{\pcal}_{f,q,D,M} \leq \Phi$. If $D$ is of bounded type, we get $\Phi=\pcal_{q,f,D,M}\leq \widetilde{\pcal}_{f,q,D,M} \leq \Phi$, so that $\pcal_{q,f,D,M} = \widetilde{\pcal}_{f,q,D,M}=\Phi$ in this case. 
\end{rem}

\begin{rem}
The case $M=+\infty$ appears to be more delicate. This case was studied in \cite{NH} and \cite{SiTo}. The discussion preceding Theorem 23 of the latter paper suggests that the \textit{unbounded \psh\ hull} $\cld^{\wedge}$ of the closure $\cld$ of the domain $D$ may be relevant. It is defined by
$\cld^{\wedge} := \bigcup_{r>0} \big(\cl{D \cap B_r(0)}\big)^{\wedge}$, where $K^{\wedge}$ denotes \psh\ hull of a compact set $K \subset \cld$. This motivates to define the \textit{$q$-hull} of $K$ in $\cld$ via
\[
K^{\wedge}_q := \{ z \in \cld : \psi(z) \leq \max_{K} \psi \ \text{for\ all}\ \psi \in \PSH_q(D)\cap\USC(\cld)\}.
\]
Then $K^\wedge_0=K^\wedge$, and $\pcal_{f,q,D,M} \leq \max_{\partial D \cap \cl{B_r(0)}} f \leq M$ holds true at least on the \textit{unbounded $q$-hull} $\cld^{\wedge}_q := \bigcup_{r>0} \big(\cl{D \cap B_r(0)}\big)^{\wedge}_q$ for every sufficiently large $r>0$ such that $D \cap B_r(0)\neq \emptyset$. More precisely, $\cld^{\wedge}_q \subseteq \cld \setminus E_\infty$ (see Proposition~\ref{prop-q-envelope}). In general, we have $\cld^{\wedge}_q \subsetneq \cld$, but it is not immediate that $\pcal_{f,q,D,M}$ is \cont\ on~$\cld^{\wedge}_q$, as pointed out in \cite{SiTo} in the case $q=0$.
\end{rem}

\section*{Declerations}

This study was funded by the DAAD–NRF Scientific Exchange Program (2016). 


\bibliographystyle{alpha}

\vspace{1cm}

Thomas Pawlaschyk\\
University of Wuppertal\\
School of Mathematics and Natural Sciences\\
Gau{\ss}str. 20, 42119 Wuppertal, Germany\\
\href{mailto:pawlaschyk@uni-wuppertal.de}{\texttt{pawlaschyk@uni-wuppertal.de}}\\
\texttt{\href{https://orcid.org/0009-0004-0494-3273}{https://orcid.org/0009-0004-0494-3273}}\\

\end{document}